\documentclass[11pt]{article}
\usepackage[a4paper,margin=1in]{geometry}
\usepackage[T1]{fontenc}
\usepackage[utf8]{inputenc}
\usepackage{lmodern}
\usepackage{amsmath,amssymb,amsthm,mathtools}
\usepackage{array}
\usepackage{hyperref}
\usepackage{xurl}
\usepackage{microtype}
\usepackage{booktabs,tabularx}
\usepackage{enumitem}
\usepackage{xcolor}
\newcommand{\GG}[1]{{#1}}
\newcommand{\VV}[1]{{#1}}
\newcommand{\YZ}[1]{{#1}}

\hypersetup{
	colorlinks=true,
	linkcolor=blue,
	citecolor=blue,
	urlcolor=blue
}

\newtheorem{theorem}{Theorem}
\newtheorem{problem}[theorem]{Problem}
\newtheorem{proposition}[theorem]{Proposition}
\newtheorem{lemma}[theorem]{Lemma}
\newtheorem{corollary}[theorem]{Corollary}
\newtheorem{conjecture}[theorem]{Conjecture}
\newtheorem{claim}[theorem]{Claim}

\newtheorem{obs}[theorem]{Observation}
\theoremstyle{definition}

\theoremstyle{remark}
\newtheorem{remark}[theorem]{Remark}

\newcommand{\supp}{\operatorname{supp}}
\newcommand{\rank}{\operatorname{rank}}
\newcommand{\dotcupop}{\mathbin{\dot\cup}}
\newcommand{\foe}{f_{\mathrm{oe}}}

\newcommand{\fo}{f_o}

\newcommand{\E}{\mathbb E}

\newcommand{\F}{\mathbb{F}}

\newcommand{\1}{\mathbf{1}}
\DeclareMathOperator{\rk}{rank}
\DeclareMathOperator{\spn}{span}

\title{Bounds on Odd and Odd-Even Induced Subgraphs}
\author{Qiwen Guo\thanks{Department for Computing, Security and Mathematics, Royal Holloway University of London, UK} \hspace{2mm} 
	Gregory Gutin\thanks{Corresponding author. Department for Computing, Security and Mathematics, Royal Holloway University of London, UK, {\tt g.gutin@rhul.ac.uk}}
	\hspace{2mm} Yiming Hao\thanks{School of Mathematical Sciences and LPMC, Nankai University, PR China} 
	\hspace{2mm}  Yongtang Shi\thanks{Center for Combinatorics and LPMC, Nankai University, PR China } \hspace{2mm}
	Yong Zhang\thanks{Shenzhen Institutes of Advanced Technology, Chinese Academy of Sciences, PR China} \hspace{2mm}
	Yacong Zhou\thanks{Shenzhen Institutes of Advanced Technology, Chinese Academy of Sciences, PR China}
}

\date{}

\begin{document}
	\maketitle
	
	\begin{abstract}
		Let $G$ be an $n$-vertex graph and let $\ell:V(G)\to\mathbb{F}_2$
		prescribe degree parities. A set $S\subseteq V(G)$ is
		$\ell$-admissible if every $v\in S$ has degree congruent to $\ell(v)$
		modulo $2$ in $G[S]$. Let $h_\ell(G)$ be the maximum order of an
		$\ell$-admissible set, set $f_{\mathrm{oe}}(G):=\min_\ell h_\ell(G)$,
		and write $f_o(G):=h_{\mathbf{1}}(G)$, where $\mathbf{1}(v)=1$ for every $v\in V(G).$
		
		We prove three main results for graphs without isolated vertices. First,
		by extending Zeng's odd-cut method to arbitrary parity prescriptions and
		introducing a one-sided completion lemma, we show that
		$h_\ell(G)\ge 2n/11$ for every $\ell$. Consequently,
		$f_{\mathrm{oe}}(G)\ge 2n/11$, improving the previous bound $2n/21$.
		Second, for bipartite graphs we derive lower bounds on $f_o(G)$ in terms
		of the $\mathbb{F}_2$-rank of the bipartite adjacency matrix and combine
		them to obtain
		\[
		f_o(G)\ge \left(\frac14+\frac1{256}\right)n=\frac{65}{256}n.
		\]
		Thus, in the bipartite case, the factor $2$ in Scott's bound
		$f_o(G)\ge n/(2\chi(G))$ can be replaced by $128/65<2$.
		Finally, a fourth-moment argument gives, for
		$\alpha=\alpha(G)\ge2$ and $\delta(G)\ge 1$,
		\[
		f_o(G)\ge \frac{\alpha}{2}+\frac{\log_3\alpha}{8}
		-\frac14\log_3\log_3\sqrt{\alpha}.
		\]
		We also construct bipartite graphs satisfying
		\[
		f_o(G)\le \frac{\alpha(G)}2+\log_2\!\bigl(\alpha(G)+1\bigr)+\frac12,
		\]
		showing that the logarithmic additive improvement over Scott's bound
		$f_o(G)\ge\alpha(G)/2$ has the optimal order of magnitude.
	\end{abstract}
	
	\section{Introduction}
	
	All graphs in this paper are without loops and multiple edges.  Let $G$ be a graph with $n$ vertices.
	For $S\subseteq V(G)$ and
	$v\in V(G)$, write
	$
	d_S(v,G):=|N_G(v)\cap S|.
	$
	When the graph is clear, we abbreviate this to $d_S(v)$.
	Let $\ell:V(G)\to\F_2$ be a vertex labelling, interpreted as a prescribed
	degree parity.  A set $S\subseteq V(G)$ and \GG{subgraph $G[S]$ are} called \emph{$\ell$-admissible}
	if
	$
	d_S(v,G)\equiv \ell(v)\pmod 2
	\quad\text{for every }v\in S.
	$
	Define
	$
	h_\ell(G):=
	\max\bigl\{|S|:S\subseteq V(G)\text{ is $\ell$-admissible}\bigr\}
	\mbox{  and }
	\foe(G):=\min_{\ell:V(G)\to\F_2} h_\ell(G).
	$
	For $\ell\equiv 1$, \GG{we call $\ell$-admissible sets and graphs {\em odd} and} 
	define $\fo(G):= h_\ell(G).$ \GG{Thus, the notion of $\ell$-admissibility
	generalises the all-odd degree case
	to the odd-even degree one, and it was first introduced by Gutin and Yeo \cite{GutinYeo2022}.} 
	
	Ai et al. \cite{AiEtAl2026} proved that if the minimum degree $\delta(G)\ge 1$, then $ \foe(G)\ge n/10004$. 
	Using a subgraph counting approach, Gutin, Hao and Zhou~\cite{GHZ} improved the bound to
	$
	\foe(G)\ge 2n/21.
	$
	Zeng~\cite{Zeng} introduced a new approach using odd-cut sets and improved the last bound to  $\fo(G)\ge n/5$
	for the case of $\ell\equiv 1$. In Section \ref{sec:labelled}, we generalize odd-cut sets to labelled-cut sets and prove 
	that \YZ{$\foe(G)\ge 2n/11$.} A ``weaker'' bound for the general case is not surprising as there are graphs $G$ for which 
	$\fo(G)=2\foe(G)$ \cite{AiEtAl2026}. Another difference between Zeng~\cite{Zeng} and Section \ref{sec:labelled} is that
	we use a one-sided completion lemma: for a labelled-cut set, prescribed cross parities on one
	side can be completed, by solving a binary linear system, to a two-sided labelled cut of at least the same order as that side. \YZ{The same lemma is used again in Section~\ref{sec:k=2}, and we therefore state and prove it separately in Section~\ref{sec:completion}.}
	
	Solving a long-standing conjecture, Ferber and Krivelevich \cite{FK2022} proved the first linear bound $\fo(G)\ge n/10000$ on $\fo(G)$. They mentioned
	that the coefficient $10^{-4}$ was probably far from optimal. Indeed very recently, Gutin, Hao and Zhou~\cite{GHZ}  and Zeng~\cite{Zeng} significantly improved the 
	Ferber-Krivelevich bound, see the bounds in the previous paragraph. Another long-standing conjecture on the topic is that by Scott \cite{Scott1992}:
	\begin{conjecture}\label{conj:Scott}
		Every graph $G$ with $\delta(G)\geq 1$ satisfies
		$
		\fo(G) \geq n/\chi(G).
		$
	\end{conjecture} 
	Wang and Wu  \cite{WangWu2024} proved that Conjecture \ref{conj:Scott} is false for bipartite graphs. Ning \cite{Ning2026} provided further counterexamples to the conjecture, also only for bipartite graphs. Note that Conjecture \ref{conj:Scott} was an attempt by Scott to strengthen his own result: $\fo(G) \geq n/(2\chi(G)).$  
	This leads to our reformulation of Conjecture \ref{conj:Scott} as an open problem:
	\begin{problem}\label{prob:tk}
		For $k\in \mathbb{N}$, determine $t_k$, where
		$$t_k:=\inf\{t \ge 1:\ \fo(G) \geq n/(tk)  \mbox{ for every graph $G$ with $\chi(G)=k$ and $\delta(G)\ge 1$} \}.$$
	\end{problem}
	Zeng's bound $\fo(G)\ge n/5$ implies that $t_k=1$ for every $k\ge 5$ and $t_i<2$ for $i=3$ and 4. In Section \ref{sec:k=2}, we prove that $t_2<2$ as well. 
	To prove the last bound, we use lower bounds on $\fo(G)$ which involve the binary rank of the bipartite adjacency matrix of $G.$ To the best of our knowledge, using ranks in such bounds has not been done before.  We also believe that such bounds can be useful in future research. 
	
	Scott's bound $\fo(G) \geq n/(2\chi(G))$ immediately follows from his other bound $\fo(G)\ge \alpha(G)/2$. In Section \ref{sec:alpha}, using the fourth moment method, 
	we improve the bound $\fo(G)\ge \alpha(G)/2$ to  
	$$f_o(G)\ \ge\ \frac{\alpha(G)}{2}+\frac{\log_3(\alpha(G))}{8}
		-\frac{\log_3\log_3(\sqrt{\alpha(G)})}{4}
		$$ for $\delta(G)\ge 1$ and $\alpha(G)\ge 2.$
 We also construct bipartite graphs $G$ with
		$f_o(G)\le \alpha(G)/2+\log_2(\alpha(G)+1)+1/2$ which implies that the additive
		logarithmic term in the improved bound is of optimal order.

	\section{A One-sided Completion Lemma}\label{sec:completion}
	
	Section~\ref{sec:labelled} and Section~\ref{sec:k=2} both rest on the same mechanism, which we isolate here.  Suppose that, in a bipartite graph, every vertex of one part has the prescribed
	parity of degree into the other part.  The next lemma shows that, by solving a binary linear system, one can pass to an induced subgraph in which the prescribed parity is attained at
	every vertex, and that nothing is lost in the order: the subgraph has at least as many vertices as that part. Lemma~\ref{lem:completion} is the main new mechanism in our
	proof of the main result of Section~\ref{sec:labelled} compared to Zeng's odd-cut argument~\cite{Zeng}. Zeng constructs cuts for which both parts already have the required cross
	parity, whereas here only the vertices of $X$ are initially controlled. 
	
	Throughout this section $B=(X,Y;F)$ is a bipartite graph with parts $X$ and $Y$, and all degrees
	are taken in $B$: for $v\in X\cup Y$ and $S\subseteq X\cup Y$ we abbreviate $d_S(v,B)$ to
	$d_S(v)=|N_B(v)\cap S|$.  In particular $d_Y(x)$ is the degree of $x\in X$ and $d_X(y)$ is the
	degree of $y\in Y$.
	
	\begin{lemma}\label{lem:completion}
		Let $B=(X,Y;F)$ be a bipartite graph and let $\ell:X\cup Y\to\F_2$ be a labelling such that
		$
		d_Y(x)\equiv \ell(x)\pmod 2
		\text{  for every }x\in X.
		$
		Then there are sets $X'\subseteq X$ and $Y'\subseteq Y$ with $|X'|+|Y'|\ge |X|$ such that
		$
		d_{Y'}(x)\equiv \ell(x)\pmod 2
		$
		for every $x\in X'$ and
		$
		d_{X'}(y)\equiv \ell(y)\pmod 2
		$
		for every $y\in Y'$.
	\end{lemma}
	
	\begin{proof}
		Let $A$ be the $X\times Y$ bipartite adjacency matrix of $B$, viewed over
		$\F_2$, and let $\ell_X$ denote the column vector of labels on $X$.  For $S\subseteq Y$ and
		$x\in X$ we have $(A\1_S)_x\equiv d_S(x)\pmod 2$. Thus, the assumption of this lemma translates to $A\mathbf 1_Y=\ell_X$.
		Let $Y'\subseteq Y$ be an inclusion-minimal set such that
		\begin{equation}\label{3.2}
			A\mathbf 1_{Y'}=\ell_X.
		\end{equation}
        Let $A_{Y'}$ be the submatrix formed by the columns indexed
		by $Y'$. If $\ell_X=0$, then $Y'=\emptyset$ and therefore we are done by setting $X':=X$ and $Y':=\emptyset$. 
		
		Thus, we may assume that $\ell_X\neq 0$ and therefore $Y'\neq \emptyset$. Then, the columns of $A_{Y'}$ are linearly independent, i.e., $\rank(A_{Y'})=|Y'|$. Indeed, if the columns indexed by a nonempty set $Y''\subseteq Y'$ summed to zero (therefore $Y''\subsetneq Y'$ as $\ell_X\neq 0$), then
		\[
		A\mathbf 1_{Y'\setminus Y''}
		=A\mathbf 1_{Y'}-A\mathbf 1_{Y''}
		=\ell_X,
		\]
		contrary to the minimality of $Y'$.  
		
		Let $\ell_{Y'}$ be the vector of labels on $Y'$.  Since $A_{Y'}^{\top}$ has full row
		rank $|Y'|$, the system
		\begin{equation}\label{3.4}
			A_{Y'}^{\top} z=\ell_{Y'} 
		\end{equation}
		has a solution $z_0\in\F_2^X$.  Its solution set is therefore $z_0+\ker A_{Y'}^{\top}$. In addition, by rank-nullity, we have
	\[
			\dim \ker A_{Y'}^{\top}=|X|-\rank (A_{Y'}^{\top})=|X|-|Y'| . \]
		Let $d=|X|-|Y'|$.  If $d=0$, set $I:=\emptyset$ and $z:=z_0$.  Otherwise let $\psi_1,\dots,\psi_d$ be a basis of $\ker A_{Y'}^{\top}$ and let $A'$ be the $d\times |X|$ matrix with rows $\psi_1,\dots,\psi_d$.  As row rank equals column rank, $A'$ has $d$ linearly independent columns. Let $I\subseteq X$ be the set of the corresponding coordinates, so that the $d\times d$ submatrix $A''$ of $A'$ formed by these columns is invertible.  Hence the restriction map $\ker A_{Y'}^{\top}\to\F_2^{I}$, $\psi\mapsto \psi|_I$, is a bijection, and in particular there is $w\in \ker A_{Y'}^{\top}$ with $w|_I=\1_I+z_0|_I$.  Let $z:=z_0+w$. Then, $z$ solves~\eqref{3.4} and $z|_I=\1_I$.
		
		Let $X'=\supp (z)$.  In either case $I\subseteq X'$, and therefore
		\[
			|X'|\ge |I|=|X|-|Y'|, 
		\]
		so that $|X'|+|Y'|\ge |X|$.  Finally, for $x\in X'$, equation~\eqref{3.2} gives
		$
		d_{Y'}(x)\equiv \ell(x)\pmod 2,
		$
		and for $y\in Y'$, the coordinate of~\eqref{3.4} indexed by $y$ gives
		$
		d_{X'}(y)\equiv \ell(y)\pmod 2.
		$
	\end{proof}
	
	\section{Labelled Cuts}\label{sec:labelled}
	
	We first generalize Zeng's odd-cut notion, which is a special case of our notion below when $\ell\equiv 1$. Let $G$ be a graph with a labelling $\ell:V(G)\to\F_2$.  A set $S\subseteq V(G)$ is an \emph{$\ell$-cut set} if it has a partition $S=P\dotcupop Q$ such that $d_Q(v) \equiv \ell(v)\pmod 2 \mbox{ for all } v\in P \mbox{ and } d_P(v) \equiv \ell(v)\pmod 2 \mbox{ for all } v\in Q$. The two parts are allowed to be empty.

	We use the following classical theorem (see, for example,
	Lov\'asz~\cite[Problem~5.17]{Lovasz}).
	
	\begin{theorem}[Gallai's even-partition theorem]\label{thm:gallai}
		Every graph $H$ has a partition $V(H)=U\dotcupop W$ such that both $H[U]$
		and $H[W]$ have all degrees even.
	\end{theorem}
	
	For $\ell\equiv 1$, Theorem~\ref{thm:labelled-cut} below is Zeng's odd-cut theorem
	\cite[Theorem~1.2]{Zeng}, and its proof is a modification of the proof of \cite[Theorem~1.2]{Zeng}.
	We provide the proof for the sake of paper's self-containment. 
	
	\begin{theorem}\label{thm:labelled-cut}
		Let $G$ be a graph with a labelling $\ell:V(G)\to\F_2$, and let $S$ be an
		$\ell$-cut set.  Then there is a partition
		$
		S=S_0\dotcupop S_1
		$
		such that both $S_0$ and $S_1$ are $\ell$-admissible. 
	\end{theorem}
	
	\begin{proof}
		Choose a witnessing partition $S=P\dotcupop Q$.  Apply
		Theorem~\ref{thm:gallai} to $G[S]$ and obtain a partition
		$S=U\dotcupop W$ such that $G[U]$ and $G[W]$ are even.  Set
		$S_0=(U\cap P)\cup(W\cap Q)$
		and 
		$ S_1=(W\cap P)\cup(U\cap Q).$
		Clearly, $S=S_0\dotcupop S_1$.
		
		Fix $i\in\{0,1\}$ and $v\in S_i$.  Let $A\in\{U,W\}$ be the part containing
		$v$, and let $B\in\{P,Q\}$ be the part not containing $v$.  By the
		definition of $S_i$,
		$
		S_i=A\mathbin{\triangle}B,
		$
		where $\triangle$ denotes symmetric difference. 
		Therefore, modulo $2$,
		$
		d_{S_i}(v,G)
		\equiv d_A(v,G)+d_B(v,G)
		\equiv 0+\ell(v)
		\equiv \ell(v).
		$
		The first term is even because $G[A]$ is even, and the second has parity
		$\ell(v)$ because $P\dotcupop Q$ witnesses that $S$ is an $\ell$-cut set.
		Thus both $S_0$ and $S_1$ are $\ell$-admissible.  \end{proof}

	The labelled odd-cut theorem is useful only after a large $\ell$-cut set has been found. For arbitrary labels, imposing the correct cross parity on both parts at once is inconvenient, and this is where Lemma~\ref{lem:completion} is applied: given disjoint $P,Q\subseteq V(G)$ such that every $q\in Q$ has the prescribed parity of degree into $P$, apply the lemma to the
    bipartite graph $G[P,Q]$, whose edges are the edges of $G$ between $P$ and $Q$, with $X:=Q$ and $Y:=P$, to obtain the resulting sets $X'\subseteq Q$ and $Y'\subseteq P$ \VV{that} satisfy $|X'|+|Y'|\ge |Q|$ and are the two parts of an $\ell$-cut set $Y'\dotcupop X'$ of $G$.  Theorem~\ref{thm:labelled-cut} then gives $h_\ell(G)\ge (|X'|+|Y'|)/2\ge |Q|/2$, which is the following corollary.
	
	\begin{corollary}\label{cor:one-sided}
		Let $P,Q\subseteq V(G)$ be disjoint and suppose that
		$
		d_P(q,G)\equiv \ell(q)\pmod 2
		\mbox{ for all } q\in Q.
		$
		Then
		$
		h_\ell(G)\ge |Q|/2.
		$
	\end{corollary}
	
	
	Let $\gamma(G)$ denote the domination number of $G$.  We use the following
	result of Bollob\'as and Cockayne~\cite{BC}; see also~\cite{Haynes}.
	
	\begin{theorem}[Bollob\'as--Cockayne]\label{thm:BC}
		Every graph $G$ without isolated vertices has a minimum dominating set
		$D$ such that, for each $u\in D$, there exists a vertex
		$p_u\in V(G)\setminus D$ satisfying
		\[
		N_G(p_u)\cap D=\{u\}.
		\]
	\end{theorem}
	
	The vertices $p_u$ are external private neighbours and are necessarily distinct.
	
\YZ{

\begin{theorem}\label{thm:pointwise}
	Let $G$ be an $n$-vertex graph without isolated vertices and let
	$\ell:V(G)\to\F_2$. Let $D$ be a minimum dominating set as in
	Theorem~\ref{thm:BC}. For each $u\in D$, let $P_u:=\{v\in V(G)\setminus D:N_G(v)\cap D=\{u\}\}$, and $T:=\{u\in D:\ell\text{ is not constant on }P_u\}$. Then
	\[
		h_\ell(G)\ge
		\max\left\{
		\frac{|D|}{2}+\frac{|T|}{4},
		\frac{4n-2|D|-|T|}{18}
		\right\}.
	\]
	In particular, $h_\ell(G)\ge 2n/11$, and therefore $\foe(G)\ge \lceil 2n/11\rceil$.
\end{theorem}

\begin{proof}
    By Theorem~\ref{thm:BC}, the sets $P_u$ are nonempty and pairwise disjoint. Let $B:=\{u\in D:\ell(u)=1\mbox{ and } P_u\cap\ell^{-1}(1)\neq\emptyset\}$. We first construct an $\ell$-cut set $S=L\dotcupop R$ as follows. For any $u\in D$,
		\begin{itemize}
				\item If $\ell(u)=0$, put only $u$ in $L$; 
				\item If $\ell(u)=1$ and $u\in B$, put $u$ in $L$ and one $1$-labelled vertex of $P_u$ in $R$;
				\item If $\ell(u)=1$ and $u\notin B$, put only one $0$-labelled vertex of $P_u$ in $R$. 
		\end{itemize}
	The private-neighbour property ensures that $S$ is an $\ell$-cut set with $|S|=|D|+|B|$. Thus, by Theorem~\ref{thm:labelled-cut}, we have that
	\begin{equation}\label{eq:n1}
		2h_\ell(G)\ge |D|+|B|.
	\end{equation}
	
	   On the other hand, let $K=D\setminus T=\{u\in D:\ell\text{ is constant on }P_u\}$. For each $u\in K$, choose $p_u\in P_u$, and for each $u\in T$, choose $p_u^0,p_u^1\in P_u$ with $\ell(p_u^i)=i$ for $i\in\{0,1\}$.
		Let $P$ consist of all these chosen private neighbours,
		so $|P|=|D|+|T|$, and let $W$ be the set of $0$-labelled vertices of $D\cup P$. Then, for each $u\in K\setminus B$, the pair $\{u,p_u\}$ contains at least one vertex of $W$. For each $u\in T\setminus B$, the triple $\{u,p_u^0,p_u^1\}$ contains two vertices of $W$, while for each $u\in T\cap B$, it contains one. Since these pairs and triples are disjoint and $K\cup T=B\cup (D\setminus B)= D$
		\[|W|\ge |K\setminus B|+2|T\setminus B|+|T\cap B|=|K\cup T|-|B|+|T\setminus B|+|T\cap B|=|D|+|T|-|B|.\]
	By Theorem~\ref{thm:gallai}, $G[W]$ has an induced subgraph of order at least $|W|/2$ in which all degrees are even. Since all vertices of $W$ have label $0$, this subgraph is $\ell$-admissible. Thus $2h_\ell(G)\ge |D|+|T|-|B|$. Combining it with \eqref{eq:n1}, we obtain
	\begin{equation}\label{eq:n2}
		4h_\ell(G)\ge 2|D|+|T|,
	\end{equation}
	which gives the first lower bound. 
	
	We now prove that $18h_\ell(G)\ge4n-2|D|-|T|$. Choose a random set $Z\subseteq D$, independently over the vertices, according to the following.
	\begin{itemize}
		\item If $u\in T$, then $\Pr(u\in Z)=\frac12$;
		\item If $u\in K$ and $\ell(p_u)=1$, then $\Pr\bigl(u\in Z\bigr)
		=\frac23$;
	\item If $u\in K$ and $\ell(p_u)=0$, then $\Pr\bigl(u\in Z\bigr)
		=\frac13$.
	\end{itemize}

	     Let $Q_Z:=\{v\in V(G)\setminus D:
		d_Z(v,G)\equiv\ell(v)\pmod2\}$.
		For each $u\in K$, the chosen vertex $p_u$
		belongs to $Q_Z$ with probability $2/3$.
		For each $u\in T$, exactly one of
		$p_u^0,p_u^1$ belongs to $Q_Z$, regardless of $Z$.
		Therefore
		\[\mathbb E|Q_Z\cap P|=\frac23|K|+|T|.\]
	
	We claim that every $v\in V(G)\setminus(D\cup P)$ satisfies $\Pr(v\in Q_Z)\ge4/9$. Indeed, since $D$ is a dominating set,
	$N_G(v)\cap D\neq\emptyset$. If $v$ has a neighbour in $T$, then conditioning on all other random choices shows that $\Pr(v\in Q_Z)=1/2$. Otherwise, $N_G(v)\cap D\subseteq K$.
	If $N_G(v)\cap D=\{u\}$, then $v\in P_u$ and hence as $u\in K$, we have $\ell(v)=\ell(p_u)$, giving $\Pr(v\in Q_Z)=2/3$. Finally, suppose $|N_G(v)\cap D|\ge2$ and let $X_u$ be the random variable indicates whether $u$ is in $Z$, i.e., $X_u=1$ if $u\in Z$ and $X_u=0$ otherwise. Then, $\frac{1+(-1)^{\ell(v)+\sum_{u\in N_G(v)\cap D}X_u}}{2}$ is the indicator for whether $v\in Q_Z$ and therefore, as $X_u$s are independent, $\E\left[(-1)^{X_u}\right]=-\Pr(u\in Z)+(1-\Pr(u\in Z))=1-2\Pr(u\in Z)$ and $|1-2\Pr(u\in Z)|=3^{-1}$, we have that
	\begin{eqnarray*}
		\Pr(v\in Q_Z)&=&\E\left[\frac{1+(-1)^{\ell(v)+\sum_{u\in N_G(v)\cap D}X_u}}{2}\right]\\
		&=&\frac12 +\frac{(-1)^{\ell(v)}}{2}
		\prod_{u\in N_G(v)\cap D}
		(1-2\Pr(u\in Z))\\
		&\ge&\frac{1-3^{-|N_G(v)\cap D|}}2
		\ge\frac49.
	\end{eqnarray*}
	
		Since $|V(G)\setminus(D\cup P)|=n-2|D|-|T|$, linearity of expectation gives
	\[
			\mathbb E|Q_Z| \ge \frac23|K|+|T|+\frac49(n-2|D|-|T|)=\frac{4n-2|D|-|T|}{9}.
	\] 
	Thus, some choice of $Z$ satisfies
		$|Q_Z|\ge(4n-2|D|-|T|)/9$ and therefore, by Corollary~\ref{cor:one-sided}, we have that
	\begin{equation}\label{eq:n3}
		18h_\ell(G)\ge4n-2|D|-|T|, 
	\end{equation}
	which gives the second lower bound.

By adding \eqref{eq:n2} and \eqref{eq:n3}, we obtain $22h_\ell(G)\ge4n$, and hence $h_\ell(G)\ge2n/11$, completing the proof.
\end{proof}
}
	Note that in the all-odd case, Zeng takes all of $D$ and one private neighbour for each $u\in D$. The resulting $2\gamma(G)$ vertices form an odd-cut set and give an odd induced subgraph of order at least $\gamma(G)$.  Under an arbitrary prescription, taking all of $D$ no longer gives the correct cross parity at a private neighbour labelled $0$. The price of arbitrary labels is that the guaranteed contribution is \YZ{roughly} $\gamma(G)/2$ rather than $\gamma(G)$. At the end, this leads to $2n/11$ rather than $n/5.$

	\section{Proving $t_2<2$ Using Rank Bounds}\label{sec:k=2}
	Let $G$ be a bipartite graph with parts $X$ and $Y$, $|X|+|Y|=n$. Let $M\in\mathbb{F}_2^{X\times Y}$ denote its bipartite adjacency matrix and $k=\rk_{\mathbb{F}_2}(M)$. Write
	$$
	\mathcal{X}=\{Mt : t\in\mathbb{F}_2^{Y}\}\subseteq\mathbb{F}_2^{X},
	\qquad
	\mathcal{Y}=\{M^{\top}s : s\in\mathbb{F}_2^{X}\}\subseteq\mathbb{F}_2^{Y},
	$$
	so that $\dim\mathcal{X}=\dim\mathcal{Y}=k$.
	
	We use the following notation throughout. For $S\subseteq X$ we write $\mathbf{1}_S\in\mathbb{F}_2^{X}$ for its characteristic vector, and similarly for subsets of $Y$, and thus $\supp(\mathbf{1}_S)=S$. All inner products $\langle u,v\rangle=\sum_i u_iv_i$ are taken over $\mathbb{F}_2$, and $W^{\perp}=\{v:\langle v,w\rangle=0\text{ for all }w\in W\}$, the ambient space being clear from the context. For $A\subseteq X$ we let $M_A$ be the submatrix of $M$ formed by the rows indexed by $A$, and for $B\subseteq Y$ we let $M_B$ be the submatrix formed by the columns indexed by $B$ (no ambiguity arises, as $A\subseteq X$ and $B\subseteq Y$).
	
	\subsection{Bound $\fo(G)\ge n/2-2k$}
	
	\begin{obs}\label{obs:oddpairs}
		Let $S\subseteq X$ and $T\subseteq Y$. $G[S\cup T]$ is odd if and only if
		$\supp(\mathbf{1}_S)\subseteq\supp(M\mathbf{1}_T)$ and
		$\supp(\mathbf{1}_T)\subseteq\supp(M^{\top}\mathbf{1}_S)$.
	\end{obs}
	
	\begin{lemma}\label{lem:duality}
		$(\ker M)^{\perp}=\mathcal{Y}$ and
		$(\ker M^{\top})^{\perp}=\mathcal{X}$.
	\end{lemma}
	
	\begin{proof}
		For all $s\in\mathbb{F}_2^{X}$ and $u\in\ker M$, exchanging the order of
		summation,
		\[
		\langle M^{\top}s,\,u\rangle
		=\langle s,\,Mu\rangle=0,
		\]
		so $\mathcal{Y}\subseteq(\ker M)^{\perp}$. By rank--nullity
		we have $\dim\ker M=|Y|-k$, hence $\dim(\ker M)^{\perp}\leq k=\dim\mathcal{Y}$, and therefore the inclusion is an equality.
		The second statement is the first applied to $M^{\top}$.
	\end{proof}
	
	\begin{lemma}\label{lem:boundintwoways}
		Let $x^{*}\in\mathcal{X}$ and $y^{*}\in\mathcal{Y}$ be of maximum support size among all vectors in $\mathcal{X}$ and $\mathcal{Y}$. Let $A=\supp(x^*)$, $B=\supp(y^*)$, $k_A=\rk M_A$ and
		$k_B=\rk M_B$. Then
		\[
		\fo(G)\ \ge\ |\supp(x^*)|+|\supp(y^*)|-(k_A+k_B).
		\]
	\end{lemma}
	\begin{proof}
		If $k=0$ then $M=0$, so $x^{*}=y^{*}=0$ and $A=B=\emptyset$. Thus $|\supp(x^*)|=|\supp(y^*)|=k_A=k_B=0$, the asserted inequality reads $\fo(G)\ge 0$, and we are done. Thus, we assume $k\geq 1$ and therefore $x^*, y^*\neq 0$. We first show that the following claim holds.
		\begin{claim}\label{cl:key}
			There exists $t\in\mathbb{F}_2^{Y}$ ($s\in\mathbb{F}_2^{X}$, resp.) with $Mt=x^*$ ($M^{\top} s=y^*$) such that  (i) $\supp(t)\subseteq B$ ($\supp(s)\subseteq A$); (ii) $|\supp(t)|\geq |B|-k_B$ ($|\supp(s)|\geq |A|-k_A$).
		\end{claim}
		\begin{proof}
			We only show the claim for $x^*$ and $t$, as the argument for the other is similar, with $M$, $\mathcal{Y}$, $B$ replaced by $M^{\top}$, $\mathcal{X}$, $A$.
			
			(i): Let $\bar B:= Y\setminus B$. Take an arbitrary $t_1$ with $Mt_1=x^*$. If $t_1|_{\bar B}=0$ we are done and therefore we assume that $t_1|_{\bar B}\neq 0$. It suffices to find $u\in\ker M$ with $u|_{\bar B}=t_1|_{\bar B}$, as then $t=t_1+u$ satisfies $Mt=x^*$ and $t|_{\bar B}=0$. Suppose, for a contradiction, that there is no such $u$, i.e., $t_1|_{\bar B}\notin U:=\{u|_{\bar B}: u\in \ker M\}$. As $U$ is a subspace of $\mathbb{F}_2^{\bar B}$ and $t_1|_{\bar B}\notin U$, we have $\dim U<|\bar B|$, and thus by rank-nullity, $\dim U^{\perp}\ge1$ (where $U^{\perp}$ is taken inside $\mathbb{F}_2^{\bar B}$). Therefore there is a non-zero vector $v\in \mathbb{F}_2^{\bar B}$ such that $\langle v,u|_{\bar B}\rangle=0$ for all $u\in\ker M$. Let $\hat v\in\mathbb{F}_2^{Y}$ be the extension of $v$ by
			zeros. Then, as $\langle\hat v,u\rangle=\langle v,u|_{\bar B}\rangle=0$ for all $u\in\ker M$, $\hat v\in(\ker M)^{\perp}=\mathcal{Y}$ by Lemma~\ref{lem:duality}. Thus $\hat v$ is a nonzero vector of $\mathcal{Y}$ whose support is contained in $\bar B$ and hence disjoint from $\supp(y^{*})$. Since $\mathcal{Y}$ is a linear subspace, $y^{*}+\hat v\in\mathcal{Y}$, and as the two supports are disjoint,
			$\supp(y^{*}+\hat v)=\supp(y^{*})\cup\supp(\hat v)$. Hence
			\[
			\supp(y^*)\subsetneq\supp(y^{*}+\hat v),
			\]
			contradicting the maximality of $y^{*}$. This proves (i).
			
			(ii): By (i), let $t'\in \mathbb{F}_2^Y$ with $Mt'=x^*$ such that $\supp(t')\subseteq B$. Thus, $M_B(t'|_B)=x^*$. If $|B|-k_B=0$ there is nothing to prove, so assume $|B|>k_B$. As $\dim \ker M_B = |B| - k_B$, we may assume $\phi_1, \phi_2, \dots, \phi_{|B|-k_B}$ is a basis for $\ker  M_B$. Thus, as row rank equals column rank, there is a subset $I\subseteq B$ with $|I|=|B|-k_B$, such that $\phi_1|_I, \phi_2|_I, \dots, \phi_{|B|-k_B}|_I$ is a basis for $\mathbb{F}_2^{I}$, and therefore in particular there exists a $u\in\spn\{\phi_1, \phi_2, \dots, \phi_{|B|-k_B}\}$ such that $u|_I+t'|_I=\mathbf{1}_I$. Let $\hat u\in\mathbb{F}_2^{Y}$ be the extension of $u$ by zeros and $t:=t'+\hat u$. Then, $Mt=Mt'+M\hat u=Mt'+M_Bu=x^{*}+0=x^{*}$ and $\supp(t)\subseteq B$. Moreover $t|_I=t'|_I+u|_I=\mathbf{1}_I$, so $I\subseteq\supp(t)$ and thus $|\supp(t)|\ge|I|=|B|-k_B$, as required.
		\end{proof}
		
		By Claim \ref{cl:key}, there exists $t$ with
		\[
		Mt=x^{*},\qquad \supp(t)\subseteq B,\qquad |\supp(t)|\ge |B|-k_B=|\supp(y^*)|-k_B,
		\]
		and $s$ with
		\[
		M^{\top}s=y^{*},\qquad \supp(s)\subseteq A,\qquad
		|\supp(s)|\ge |A|-k_A=|\supp(x^*)|-k_A.
		\]
		
		Let $S=\supp(s)$ and $T=\supp(t)$. Both are nonempty since
		$M^{\top}s=y^{*}\ne 0$ and $Mt=x^{*}\ne 0$. Then
		\[
		\supp(s)\subseteq A=\supp(x^{*})=\supp(Mt),
		\qquad
		\supp(t)\subseteq B=\supp(y^{*})=\supp(M^{\top}s),
		\]
		so $G[S\cup T]$ is odd by Observation~\ref{obs:oddpairs}, and
		\[\fo(G)\geq |S|+|T|\ge(|\supp(x^*)|-k_A)+(|\supp(y^*)|-k_B),\]
		completing the proof.
	\end{proof}
	
	Clearly $k_A,k_B\leq k$ and on the other hand if $\delta(G)\geq 1$, by taking $x$ and $y$ uniformly at random from $\mathcal{X}$ and $\mathcal{Y}$, we can see that $\mathbb{E} |\supp(x)|=|X|/2$ and $\mathbb{E} |\supp(y)|=|Y|/2$ \VV{(indeed, as observed in the proof of Lemma~\ref{lem:ranksensitive} below, $\delta(G)\ge 1$ implies that no coordinate vanishes identically on $\mathcal{X}$ or on $\mathcal{Y}$, and a coordinate that does not vanish identically equals $1$ for exactly half of the vectors of the subspace)} and therefore $|\supp(x^*)|\geq |X|/2$ and $|\supp(y^*)|\geq |Y|/2$. Thus, we have the following corollary.
	
	\begin{corollary}\label{cor:n2minus2k}
		If $\delta(G)\geq 1$, then $\fo(G)\geq n/2-2k$.
	\end{corollary}
	
	\subsection{\texorpdfstring{\VV{Rank-sensitive bounds, including $f_o(G)\geq k$}}{Rank-sensitive bounds, including fo(G) >= k}}
	
	In this section we show that each of $|\supp(x^{*})|$ and $|\supp(y^{*})|$ is on its own a lower bound for $\fo(G)$, and that the estimates $|\supp(x^{*})|\ge|X|/2$ and $|\supp(y^{*})|\ge|Y|/2$ used above can be improved by an amount depending on $k$. Combining the resulting bound with Corollary~\ref{cor:n2minus2k} gives a uniform improvement on the bound $\fo(G)\ge n/4$. 
	\begin{lemma}\label{lem:onesided}
		Let $x^{*}\in\mathcal{X}$ and $y^{*}\in\mathcal{Y}$ be of maximum support size among all vectors in $\mathcal{X}$ and $\mathcal{Y}$. For every $t\in\mathbb{F}_2^{Y}$ we have $\fo(G)\ge|\supp(Mt)|$, and for every $s\in\mathbb{F}_2^{X}$ we have $\fo(G)\ge|\supp(M^{\top}s)|$. In particular,
		\[
		\fo(G)\ \ge\ \max\bigl\{|\supp(x^{*})|,\ |\supp(y^{*})|\bigr\}.
		\]
	\end{lemma}
	
	\begin{proof}
		We only prove the first assertion, the second being the first applied to $M^{\top}$.
		
		Let $t\in\mathbb{F}_2^{Y}$, $T:=\supp(t)$ and $R:=\supp(Mt)$. We may assume $R\neq\emptyset$.
		By the definition of $R$, every $r\in R$ satisfies $d_T(r,G)\equiv 1\pmod 2$.  As $G$ is
		bipartite with $R\subseteq X$ and $T\subseteq Y$, the induced subgraph $B:=G[R\cup T]$ has no
		edges inside $R$ and none inside $T$, so it is the bipartite graph with parts $R$ and $T$.
		Applying Lemma~\ref{lem:completion} to $B$, whose two parts $R$ and $T$ play the roles of $X$
		and $Y$, with $\ell\equiv 1$, we obtain sets $X'\subseteq R$ and $Y'\subseteq T$ with
		$|X'|+|Y'|\ge |R|$ such that every vertex of $B[X'\cup Y']$ has odd degree.  Since
		$B[X'\cup Y']=G[X'\cup Y']$, the graph $G[X'\cup Y']$ is odd, and therefore
		\[
		\fo(G)\ \ge\ |X'|+|Y'|\ \ge\ |R|\ =\ |\supp(Mt)| ,
		\]
		which completes the proof.
	\end{proof}
	
	Together with the fact that a $k$-dimensional subspace of $\mathbb{F}_2^{X}$ contains a vector of support size at least $k$, this yields a lower bound on $f_o(G)$ in terms of the rank of $M$, even without assuming that $\delta(G)\geq 1$, which is tight for $P_4$ (path on 4 vertices). 
	
	\begin{corollary}\label{cor:rank}
		$f_o(G)\geq k$. 
	\end{corollary}

	\begin{proof}
		Let $\Gamma_0$ be a $k\times|X|$ matrix whose rows form a basis of $\mathcal{X}$. As row rank equals column rank, $\Gamma_0$ has $k$ linearly independent columns. Let $I\subseteq X$ be the set of the corresponding coordinates. Thus, $(\Gamma_0)_I$ is invertible and therefore some $x\in\mathcal{X}$ satisfies $x|_I=\mathbf{1}_I$, so that $|\supp(x)|\ge|I|=k$. Thus, Lemma~\ref{lem:onesided} gives $\fo(G)\ge|\supp(x)|\ge k$.
	\end{proof}

	\begin{lemma}\label{lem:ranksensitive}
		Let $\delta(G)\ge1$ and $m:=\max\{|X|,|Y|\}$, so that $m\ge n/2$. Then, if $k\leq \frac{2m}{3}$, we have
		\[
		\fo(G)\ \ge
		\dfrac{m}{2}+\dfrac{k^{2}}{8(m-k)}. 
		\]
	\end{lemma}
	
	\begin{proof}
		As $\delta(G)\geq 1$, we have $k\ge1$ and $m\ge1$. Assume without loss of generality that $|X|\ge|Y|$, and therefore $m=|X|$. By Lemma~\ref{lem:onesided} we have $\fo(G)\ge|\supp(x^{*})|=\max_{x\in\mathcal{X}}|\supp(x)|$, so it suffices to bound $\max_{x\in\mathcal{X}}|\supp(x)|$ from below.
		
		No coordinate of $\mathcal{X}$ vanishes identically: given $x_0\in X$, choose $y_0\in N(x_0)$ (this is possible as $\delta(G)\ge1$), then $M\mathbf{1}_{\{y_0\}}\in\mathcal{X}$ has a $1$ in the coordinate $x_0$.
		
		Let $\Gamma_0$ be a $k\times m$ matrix whose rows form a basis of $\mathcal{X}$. As row rank equals column rank, $\Gamma_0$ has $k$ linearly independent columns. After ordering the elements of $X$ so that these come first, and replacing $\Gamma_0$ by $C^{-1}\Gamma_0$, where $C$ is the invertible $k\times k$ matrix formed by those columns, we may assume that $\mathcal{X}$ has a generator matrix
		\[
		\Gamma=[\,I_k\mid D\,],\qquad D\in\mathbb{F}_2^{k\times(m-k)} .
		\]
		Every column of $D$ is nonzero, since otherwise the corresponding coordinate would vanish on all of $\mathcal{X}$, contrary to the previous paragraph.
		
		Let $q\in[0,1]$ and $p=(1+q)/2$. Let $\xi=(\xi_1,\dots,\xi_k)\in\mathbb{F}_2^{k}$ have independent coordinates with $\Pr(\xi_i=1)=p$, and consider the random vector
		\[
		W:=\xi\Gamma\in\mathcal{X} .
		\]
		The first $k$ coordinates of $W$ are $\xi_1,\dots,\xi_k$, and so contribute $pk$ to $\mathbb{E}|\supp(W)|$. Let $d$ be a column of $D$ and $\ell:=|\supp(d)|\ge1$. Since the $\xi_i$ are independent with $\mathbb{E}\bigl[(-1)^{\xi_i}\bigr]=1-2p=-q$, we have $\mathbb{E}\bigl[(-1)^{\xi\cdot d}\bigr]=\prod_{i\in \supp(d)}\mathbb{E}\bigl[(-1)^{\xi_i}\bigr]=(-q)^{\ell}$ and hence
		\[
		\Pr(\xi\cdot d=1)=\frac{1-(-q)^{\ell}}{2}\ \ge\ \frac{1-q^{2}}{2},
		\]
		because $(-q)^{\ell}=-q^{\ell}\le0$ if $\ell$ is odd, while $(-q)^{\ell}=q^{\ell}\le q^{2}$ if $\ell$ is even, in which case $\ell\ge2$. Summing over the coordinates,
		\begin{equation}\label{eq:biased}
			\mathbb{E}|\supp(W)|\ \ge\ pk+(m-k)\,\frac{1-q^{2}}{2}
			\ =\ \frac{m}{2}+\frac{qk}{2}-\frac{q^{2}(m-k)}{2}.
		\end{equation}
		
		Let $q:=\dfrac{k}{2(m-k)}$ (note that $q\le1$ because $3k\le2m$). With this choice \eqref{eq:biased} gives
		\[
		\mathbb{E}|\supp(W)|\ \ge\ \frac{m}{2}+\frac{k^{2}}{4(m-k)}-\frac{k^{2}}{8(m-k)}=\frac{m}{2}+\frac{k^{2}}{8(m-k)} .
		\]
		Thus, some vector of $\mathcal{X}$ (with coordinate vector $\xi C^{-1}$ with respect to the basis given by the rows of $\Gamma_0$) has support of at least the stated size, which proves the displayed inequality.
	\end{proof}
	
	\subsection{Proof of the main result}
	
	\begin{theorem}\label{thm:quarterplus}
		Let $G$ be a bipartite graph with $\delta(G)\ge1$. Then
		\[
		\fo(G)\ \ge\ \Bigl(\frac14+\frac1{256}\Bigr)n\ =\ \frac{65n}{256}.
		\]
	\end{theorem}
	
	\begin{proof}
		Let $m:=\max\{|X|,|Y|\}$, so that $m\ge n/2$, and let
		\[
		k_0:=\frac{63n}{512} .
		\]
		We distinguish three cases.
		
		\medskip\noindent\emph{Case 1: $k\le k_0$.} By Corollary~\ref{cor:n2minus2k},
		\[
		\fo(G)\ \ge\ \frac n2-2k\ \ge\ \frac{256n}{512}-\frac{126n}{512}\ =\ \frac{130n}{512}\ =\ \frac{65n}{256} .
		\]
		
		\medskip\noindent\emph{Case 2: $k>k_0$ and $k>2m/3$.} By Corollary~\ref{cor:rank} and the fact that $m\ge n/2$,
		\[
		\fo(G)\ \ge\ k\ >\ \frac{2m}{3}\ \ge\ \frac n3\ >\ \frac{65n}{256}.
		\]
		
		\medskip\noindent\emph{Case 3: $k>k_0$ and $k\le 2m/3$.} By Lemma~\ref{lem:ranksensitive},
		\[
		\fo(G)\ \ge\ F(m,k),\qquad\text{where } F(x,y):=\frac x2+\frac{y^{2}}{8(x-y)} \quad (0\le y<x).
		\]
		Since
		\[
		\frac{\partial F}{\partial y}=\frac{y(2x-y)}{8(x-y)^{2}}\ \ge\ 0,
		\]
		for fixed $x$, the function $F(x,\cdot)$ is increasing on $[0,x)$, and therefore $F(m,k)>F(m,k_0)$. For fixed $y$, the function $F(\cdot,y)$ satisfies
		\[
		\frac{\partial F}{\partial x}=\frac12-\frac{y^{2}}{8(x-y)^{2}} ,
		\]
		which is positive precisely when $2(x-y)>y$. Taking $y=k_0$ and $x=m\ge n/2$ we get
		\[
		m-k_0\ \ge\ \frac{256n}{512}-\frac{63n}{512}\ =\ \frac{193n}{512}\ >\ \frac{k_0}{2} ,
		\]
		so $F(\cdot,k_0)$ is increasing on $[n/2,\infty)$ and hence $F(m,k_0)\ge F(n/2,k_0)$. Finally,
		\[
		F\Bigl(\frac n2,\frac{63n}{512}\Bigr)
		=\frac n4+\frac{(63n/512)^{2}}{8\cdot(193n/512)}
		=\frac n4+\frac{3969\,n}{790528} ,
		\]
		and $\dfrac{3969}{790528}>\dfrac1{256}$, because $790528/256=3088<3969$. Combining the three displays,
		\[
		\fo(G)\ >\ \frac n4+\frac{n}{256}\ =\ \frac{65n}{256} .
		\]
		
		\medskip
		In all three cases $\fo(G)\ge 65n/256$, as required.
	\end{proof}

	\begin{remark}
	No attempt was made to optimize the constant in Theorem~\ref{thm:quarterplus}: balancing Cases 1 and 3 at $k_0=(5-\sqrt{10})n/15$ yields the marginally better bound $\fo(G)\ge (4\sqrt{10}-5)n/30\approx 0.25497\,n$.
	\end{remark}

	\begin{remark} Since $65n/256=n/(2\cdot 128/65)$, we have $t_2\le 128/65<2$.
	\end{remark}
	
	\section{Improving Scott's Bound $\fo(G)\ge \alpha(G)/2$}\label{sec:alpha}
	
	\subsection{Preliminaries}
	
	Let $G$ be a (simple) graph with $\delta(G)\ge 1$, and let $X$ be a maximum independent set of $G$, i.e., $|X|=\alpha(G).$
	Since $G$ has no isolated vertices, $V(G)\setminus X$ dominates $X$.  Let $D\subseteq V(G)\setminus X$ be a minimal subset that dominates $X$.  For every $v\in D$, choose a private neighbour $p_v\in X$, that is, $N(p_v)\cap D=\{v\}.$ Let $P=\{p_v:v\in D\}$ and $X_0=X\setminus P.$
	For $U\subseteq D$, define $$R(U)=\{x\in X_0:\ |N(x)\cap U|\text{ is odd}\}.$$
	
	The following is basically what was shown by Scott \cite{Scott1992}. 
	
	\begin{lemma}\label{lem:simple-lemma2}
		Let $G$, $X$, $D$, $P$, $X_0$, and $R(U)$ be as above.  For $U\subseteq D$, define
		$$
		Q(U)=\{p_v:v\in U,\ \deg_{G[U]}(v)+|N(v)\cap R(U)|\text{ is even}\}
		$$
		and
		$
		S(U)=U\cup R(U)\cup Q(U).
		$
		Then $G[S(U)]$ is odd.  In particular,
		$\fo(G)\ge |S(U)|\ge |U|+|R(U)|.$
	\end{lemma}
	
	The following result is a corollary of the hypercontractivity theorem (see Corollary 5.1 in \cite{ODonnell2008}). 
	
	\begin{lemma}[\cite{ODonnell2008}]\label{lem:hct}
		Let $y_1,y_2,\dots,y_n$ be i.i.d. random variables taking values uniformly from $\{+1,-1\}$ and $Y(y_1,y_2,\dots,y_n)$ a multilinear polynomial \VV{in them of} degree at most $r$. Then $\mathbb{E}[Y^4]\leq 9^{r} (\mathbb{E}[Y^2])^2$.
	\end{lemma}
	
	\subsection{$\log(\alpha(G))$ bound}
	
	To prove the key lemma, Lemma \ref{lem:bound1}, of this subsection we will use the fourth-moment probabilistic method very well presented by Berger in \cite{Berger1997}. One of the key lemmas of the method is the following:
	\begin{lemma}\label{lem:fourth-moment}\cite{Berger1997}
		Let $Y$ be a real random variable with
		$\E Y=0,$ $\E Y^2=\sigma^2>0$
		and suppose that
		$
		\E Y^4\le C\sigma^4
		$
		for some $C\ge 1$. Then the probability ${\rm Pr}(Y\ge \sigma/(2\sqrt C))>0$.
	\end{lemma}

	For every $x\in X$, put $A_x=N(x)\cap D.
	$
	The sets $A_x$ are non-empty because $D$ dominates $X$.  Let
	$
	r=\max_{x\in X}|A_x|.
	$
	The following bound deals with the case when $r$ is small.
	
	\begin{lemma}\label{lem:bound1}
		For every graph $G$ with $\delta(G)\geq 1$,
		\[
		f_o(G) \;\ge\; \frac{\alpha(G)}{2} + \frac{\sqrt{\alpha(G)}}{4 \cdot 3^{\,r}}.
		\]
	\end{lemma}
	
	\begin{proof}
		Recall that $A_x = N(x) \cap D$ and $A_x \neq \emptyset$ for every $x \in X$. For every subset $U\subseteq D$, we define
		$a_1=|\{x \in X : |A_x \cap U|\text{ odd}\}|$, $a_0=|\{x \in X : |A_x \cap U|\text{ even}\}|,$ and 
		\[
		Y(U) \;:=\; -\sum_{x \in X} (-1)^{|A_x \cap U|}=a_1-a_0.
		\]
		Since $A_{p_v} = \{v\}$, a vertex $p_v\in P$ has $|A_{p_v} \cap U|$ odd
		if and only if $v \in U$. Hence $$|U| = |\{p_v : |A_{p_v} \cap U|\text{ odd}\}|.$$
		Recall that $R(U) = \{x \in X_0 : |A_x \cap U|\text{ odd}\}$. Observe that $a_1+a_0=|X|=\alpha(G).$ Thus,
		\begin{equation}\label{eq:ZY}
			Z(U) \;:=\; |U| + |R(U)| \;=\;  a_1 \;=\; \frac{\alpha(G) + Y(U)}{2}
		\end{equation}

		Let $\mathbf{u}\in \{+1,-1\}^{|D|}$ be chosen uniformly at random from all vectors in $\{+1,-1\}^{|D|}$. For every such $\mathbf{u}$, we define $U(\mathbf{u}):=\{v\in D:u_v=-1\}$. Since $(-1)^{|A_x \cap U(\mathbf{u})|} = \prod_{v \in A_x} u_v$,
		\[
		Y(U(\mathbf{u})) \;=\; -\sum_{x \in X} \, \prod_{v \in A_x} u_v
		\]
		is a multilinear polynomial of degree $\max_{x \in X} |A_x| = r$.
		
		\smallskip
		For each $x \in X$ we have
		$\mathbb{E}\, [\prod_{v \in A_x} u_v] = \prod_{v \in A_x} \mathbb{E}\,[u_v] = 0$, since $A_x \neq \emptyset$. Therefore, $\mathbb{E}\,[Y(U(\mathbf{u}))] = 0$.
		
		\smallskip
		Now we estimate the variance of $Y(U(\mathbf{u}))$. We have
		\[
		\mathbb{E}\!\left[
		\prod_{v \in A_x} u_v \prod_{w \in A_y} u_w
		\right]
		= \mathbb{E} \left[\,\!\!\prod_{v \in A_x \triangle A_y}\!\! u_v\right]
		= \mathbf{1}[\{A_x = A_y\}],
		\]
		where $\mathbf{1}[A_x = A_y]=1$ if $A_x = A_y$ and 0 otherwise. Summing over ordered pairs, we have
		\begin{equation}\label{eq:var}
			\mathbb{E}\,[Y(U(\mathbf{u}))^2]
			= |\{(x,y) \in X^2 : A_x = A_y\}|
			\;\ge\; \alpha(G).
		\end{equation}
		
		\smallskip
		As $Y(U(\mathbf{u}))$ has degree at most $r$, applying Lemma \ref{lem:hct}, we have
		\[
		\mathbb{E}\,[Y(U(\mathbf{u}))^4] \;\le\; 9^{\,r} \,(\mathbb{E}\, [Y(U(\mathbf{u}))^2])^2.
		\]
		
		\smallskip
		Apply Lemma~\ref{lem:fourth-moment} to $Y(U(\mathbf{u}))$,  with $\sigma^2 = \mathbb{E}\,[Y(U(\mathbf{u}))^2]$ and $C = 9^{\,r}$, so that $\sqrt{C} = 3^{\,r}$.
		This produces a set $U \subseteq D$ with $Y(U) \ge \sigma / (2\sqrt{C})$. Combining \eqref{eq:ZY}, \eqref{eq:var}, and this bound,
		\[
		Z(U) = \frac{\alpha(G) + Y(U)}{2}
		\;\ge\; \frac{\alpha(G)}{2} + \frac{\sqrt{\mathbb{E}\,[Y(U(\mathbf{u}))^2]}}{4 \cdot 3^{\,r}}
		\;\ge\; \frac{\alpha(G)}{2} + \frac{\sqrt{\alpha(G)}}{4 \cdot 3^{\,r}}.
		\]
		 Thus, by Lemma \ref{lem:simple-lemma2}, $f_o(G)\geq Z(U)\geq \frac{\alpha(G)}{2} + \frac{\sqrt{\alpha(G)}}{4 \cdot 3^{\,r}}$, completing the proof. 
	\end{proof}
	
	The following lemma deals with the case when $r$ is large. 
	
	\begin{lemma}\label{lem:bound2}
		For every graph $G$ with $\delta(G)\geq 1$,
		\[f_o(G) \ge \dfrac{\alpha(G)}{2} + \dfrac{r}{4}.\]
	\end{lemma}
	
	\begin{proof}
		We first show that the following claim holds.
		\begin{claim}\label{eq:exchange}
			For every $U \subseteq D$, $x_0 \in R(U)$ and $W = N(x_0) \cap U$, we have 
			\[f_o(G) \;\ge\; |U| + |R(U)| + \frac{|W| - 1}{2}.\]
		\end{claim}
		\begin{proof}
			Indeed, by Lemma~\ref{lem:simple-lemma2}, $H := G[S(U)]$ is odd, with $S(U) = U \cup R(U) \cup Q(U)$ and
			$Q(U) \subseteq \{p_v : v \in U\}$. Each $p_v$ satisfies $N(p_v) \cap D = \{v\}$,
			so inside $U \cup R(U) \cup P$ it is adjacent only to $v$; and
			$x_0 \in R(U) \subseteq X_0$, so $x_0 \notin P$ and in $H$ the vertex $x_0$ is
			adjacent exactly to $W$. Let $P_W = \{p_v : v \in W\}$,
			$a_W = |Q(U) \cap P_W|$, and set
			\[
			S' = \bigl(S(U) \setminus \{x_0\}\bigr) \,\triangle\, P_W .
			\]
			Then $G[S']$ is odd: for $v \in W$, deleting $x_0$ lowers $\deg(v)$ by $1$ and
			toggling $p_v$ changes it by $\pm 1$, restoring parity; each inserted $p_v$ has
			degree $1$; no other degree changes, since each $p_w$ $(w\in W)$ is adjacent
			only to $w$, while $x_0$ and the $p_w$ have no neighbours in
			$R(U) \cup Q(U)$ or $U \setminus W$. Hence
			\[
			|S'| = |S(U)| - 1 - a_W + (|W| - a_W) \geq |U|+|R(U)| + |W| - 1 - 2a_W .
			\]
			As $|S(U)| = |U| + |R(U)| + |Q(U)| \ge |U| + |R(U)| + a_W$, the odd subgraphs
			$H$ and $G[S']$ give
			\[
			f_o(G) \ge \max\{|S(U)|, |S'|\}
			\ge |U| + |R(U)| + \max\{a_W,\, |W| - 1 - a_W\}
			\ge |U| + |R(U)| + \frac{|W| - 1}{2},
			\]
			completing the proof.
		\end{proof}
		
		If $r = 1$, take $U = D$, then $|A_x \cap U| = 1$ for all $x \in X$, so $Z(U) := |U| + |R(U)| = \alpha(G)$ and $f_o(G) \ge \alpha(G) \ge \frac {\alpha(G)}{2} + \frac r4$.
		
		Assume $r \ge 2$. Let $x_0 \in X$ satisfying $|A_{x_0}| = r$ and choose $U \subseteq D$ randomly and uniformly subject to
		\begin{equation}\label{eq:cond}
			|A_{x_0} \cap U| \text{ odd.}
		\end{equation}
		Then \VV{$x_0\in X_0$ (indeed, $|A_{p_v}|=1$ for every $p_v\in P$, while $|A_{x_0}|=r\ge 2$) and thus} $x_0 \in R(U)$ and $W = A _{x_0}\cap U$. Therefore, Claim \ref{eq:exchange} implies
		\[
		f_o(G) \;\ge\; |U|+ |R(U)| + \frac{|A_{x_0} \cap U| - 1}{2}.
		\]
		Under \eqref{eq:cond}, for $A_x \ne A_{x_0}$ the parity of $|A_x \cap U|$ is balanced
		by the nontrivial linear condition, so $\mathbb{P}(|A_x \cap U|\text{ odd}) = \tfrac12$,
		while it is $1$ when $A_x = A_{x_0}$; hence $\mathbb{E}\,[|U|+|R(U)|] \ge \frac{\alpha(G)}{2} + \frac12$,
		as at least $x_0$ has $A_{x_0} = A_{x_0}$. Also each element \VV{$v$} of $A_{x_0}$ lies in $U$ with probability $\tfrac12$ \VV{(here $r\ge 2$ is used: under \eqref{eq:cond}, $v\in U$ if and only if $|(A_{x_0}\setminus\{v\})\cap U|$ is even, which happens with probability $\tfrac12$ as $A_{x_0}\setminus\{v\}\neq\emptyset$)}, so $\mathbb{E}\,|A_{x_0} \cap U| = \frac r2$ and
		$\mathbb{E}\bigl[\frac{|A_{x_0} \cap U| - 1}{2}\bigr] = \frac r4 - \frac12$. Hence,
		\[
		f_o(G)\geq\mathbb{E}\!\left[ |U| + |R(U)| + \frac{|A_{x_0} \cap U| - 1}{2} \right]
		\ge \frac{\alpha(G)}{2} + \frac r4,
		\]
		completing the proof. 
	\end{proof}
	\begin{theorem}\label{thm:logbound}
		For every graph $G$ with $\delta(G)\geq 1$ and $\alpha(G)\ge 2$,
		\[f_o(G) \ge \frac{\alpha(G)}{2} + \frac{\log_3(\alpha(G))}{8}-\frac{\log_3\log_3(\sqrt{\alpha(G)})}{4}.\]
	\end{theorem}
	\begin{proof}
		Let $\alpha:=\alpha(G)$. 
		Assume that $2\le\alpha\le 8$. If $r=1$ then $f_o\ge\alpha\ge\alpha/2+1$
		(as in the first case of \VV{Lemma~\ref{lem:bound2}'s} proof), and if $r\ge 2$ then
		\VV{Lemma~\ref{lem:bound2}} gives $f_o\ge\alpha/2+1/2$, while the additive term satisfies\VV{, bounding its two terms separately (the first term is increasing in $\alpha$ and the second is decreasing in $\alpha$),}
		\[
		\frac{\log_3\alpha}{8}-\frac{\log_3\log_3\sqrt\alpha}{4}
		\ \le\ \frac{\log_3 8}{8}-\frac{\log_3\log_3\sqrt 2}{4}
		\ \approx\ 0.2366+0.2625\ <\ \frac12.
		\]

		Now assume that $\alpha\ge 9$. If $r\ge \frac{\log_3 \alpha}{2}-\log_3\log_3(\sqrt{\alpha})$, we are done by Lemma \ref{lem:bound2}. If $r\le \frac{\log_3(\alpha(G))}{2}-\log_3\log_3(\sqrt{\alpha})$, then
		
		\[\frac{\sqrt{\alpha}}{ 3^{\,r}}\geq \frac{\sqrt{\alpha}}{\sqrt{\alpha}/\log_3(\sqrt{\alpha})}=\frac{\log_3(\alpha)}{2},\]
		and therefore we are done by Lemma \ref{lem:bound1}.
	\end{proof}
	
	The following examples show that our bound is tight up to a constant. 
	
	\begin{proposition}\label{prop:tight}
		For every $d \ge 1$ there is a bipartite graph $G_d$ with
		$\delta(G_d) \ge 1$ and $\alpha(G_d) = 2^d - 1$ such that
		\[
		f_o(G_d) \;\le\; \frac{\alpha(G_d)}{2} + \log_2\!\bigl(\alpha(G_d) + 1\bigr) + \frac12.
		\]
		Consequently the additive term in Theorem~\ref{thm:logbound} is of optimal order.
	\end{proposition}
	
	\begin{proof}
		\VV{Let $[d] = \{1, \dots, d\}$ and let $X = \mathbb{F}_2^d \setminus \{0\}$, where the coordinates of vectors in $\mathbb{F}_2^d$ are indexed by $[d]$.} Define $G_d$ on vertex set
		$X \cup [d]$ by joining $v \in X$ to $i \in [d]$ exactly when $v_i = 1$. For every $i\in [d]$, let $e_i\in \mathbb{F}_2^d$ be \VV{the vector whose components are all $0$ except the $i$-th, which equals $1$}. Clearly, $G_d$ is bipartite, and $\delta(G_d) \ge 1$.
		
		\VV{$X$ is an independent set of size $2^d - 1$.} Since $G_d$ is bipartite, König's theorem gives $\alpha(G_d) = |V(G_d)| - \beta(G_d)$, where $\beta(G_d)$ is the size of the maximum matching in $G_d$. As the edges $\{i, e_i\}$ $(i \in [d])$ form a matching of size $d$ and $|[d]|=d$, we have that $\beta(G_d) = d$ and
		$\alpha(G_d) = (2^d - 1 + d) - d = 2^d - 1$, attained by $X$.
		
		Suppose that $T \subseteq V(G_d)$ induces an odd subgraph with $X' = T \cap X$ and $D' = T \cap [d]$. Clearly, $D' \ne \emptyset$. For $v \in X'$, the degree of $v$ in $G_d[T]$ is odd only if the inner product over $\mathbb{F}_2$, $\langle v, \mathbf{1}_{D'} \rangle=1$.  Hence
		\[
		X' \subseteq H := \{v \in \mathbb{F}_2^d :
		\langle v, \mathbf{1}_{D'} \rangle = 1\}.
		\]
		Fix any $j \in D'$, the map $v \mapsto v + e_j$ has no fixed-point
		on $\mathbb{F}_2^d$. In addition, as $j \in D'$ it satisfies
		$\langle v + e_j, \mathbf{1}_{D'}\rangle = \langle v, \mathbf{1}_{D'}\rangle + 1$ and therefore it interchanges $H$ with its complement. Thus, $|X'|\leq |H| = 2^{d-1}$ and therefore
		\[
		|T| = |X'| + |D'| \le 2^{d-1} + d .
		\]
		
		As this holds for arbitrary $T$, $f_o(G_d) \le 2^{d-1} + d$. Since $\alpha(G_d) = 2^d - 1$ we have $2^{d-1} = \tfrac{\alpha(G_d)+1}{2}$ and
		$d = \log_2(\alpha(G_d) + 1)$, giving the stated bound. 
	\end{proof}

	\begin{corollary}
		For a graph $G$ with $\delta(G)\ge 1$, $\alpha(G)\ge 2$, and $|V(G)|\ge 2\chi(G)$, we have \[
		f_o(G)\ge
		\frac{|V(G)|}{2\chi(G)}
		+
		\frac{\log_3(|V(G)|/\chi(G))}{8}
		-
		\frac14
		\log_3\log_3
		\sqrt{\frac{|V(G)|}{\chi(G)}}.
		\]
	\end{corollary}
	
	\begin{proof}
		Note that $\alpha(G)\ge \frac{|V(G)|}{\chi(G)}$. To apply the bound of Theorem \ref{thm:logbound}, it suffices to establish when the function 
		\[
		\Phi(x)=\frac{x}{2}+\frac{\log_3 x}{8}
		-\frac{1}{4}\log_3\log_3\sqrt{x}
		\]
		is monotonically increasing.
		Write logarithms in base $e$.  Since
		\[
		\log_3\log_3\sqrt{x}
		=
		\frac{1}{\ln 3}\ln\!\left(\frac{\ln x}{2\ln 3}\right),
		\]
		we have
		\[
		\Phi'(x)=
		\frac12+\frac{1}{8x\ln 3}
		-\frac{1}{4x\ln 3\ln x}.
		\]
		For $x\ge 2$,
		\[
		\frac{1}{4x\ln3\ln x}
		\le
		\frac{1}{8\ln3\ln2}
		<\frac14,
		\]
		because $\ln3>1$ and $\ln2>1/2$.  Hence
		\(
		\Phi'(x)>\frac12-\frac14>0.
		\)
		Thus $\Phi$ is monotonically increasing on $[2,\infty)$ and thus the bound of this corollary follows. 
	\end{proof}

\section{Open Problems}
\GG{
Let us define
$$c_{\rm o}:=\inf\{\fo(G)/|V(G)|: \delta(G)\ge 1\} \mbox{ and } c_{\rm oe}:=\inf\{\foe(G)/|V(G)|: \delta(G)\ge 1\}.$$
By \cite{Zeng}, $c_{\rm o}\ge 1/5$ and we proved here that $c_{\rm oe}\ge \YZ{2/11}$. 
It is well known that $\fo(\overline{C_7})=2$, see, e.g., \cite{Caro1994} and that $2/7$ is the best
known upper bound for $c_{\rm o}$ \cite{AiEtAl2026,Zeng}. We determine $\foe(\overline{C_7})$ below.
}
{
\begin{proposition}\label{prop:lower}
$\foe(\overline{C_7})=2.$
\end{proposition}
\begin{proof}
Denote $H=\overline{C_7}.$ It is easy to check that the independence number $\alpha(H)=2$ and the clique number $\omega(H)=3.$

Since $\fo(H)=2,$\footnote{\VV{For completeness, here is a short proof that $\fo(H)=2$. Any set $S$ inducing a subgraph with all degrees odd has even order by the handshake lemma. If $|S|=6$, say $S=V(H)\setminus\{v\}$, then the two non-neighbours of $v$ retain their even degree $4$ in $H[S]$. If $|S|=4$, note that the only $4$-vertex graphs with all degrees odd are $2K_2$, $K_{1,3}$ and $K_4$; their complements within $K_4$ are $C_4$, $K_3\cup K_1$ and $\overline{K_4}$, respectively, and none of these occurs as an induced subgraph of $C_7$ on four vertices, since such a subgraph is a disjoint union of paths and $\alpha(C_7)=3$. Finally, any edge of $H$ induces an odd subgraph of order 2.}} $\foe(H)\le 2.$
It remains to prove that $\foe(H)\ge 2$, i.e.,
$h_\ell(H)\ge 2$ for every labelling $\ell:V(H)\to\{0,1\}$.
Observe how two-vertex sets of $H$ behave. If $uv\in E(H)$ then in
$H[\{u,v\}]$ both degrees are 1, so $\{u,v\}$ is $\ell$-admissible
iff $\ell(u)=\ell(v)=1$. If $uv \not\in E(H)$ then both degrees are 0,
so $\{u,v\}$ is $\ell$-admissible iff $\ell(u)=\ell(v)=0$.

Suppose, for a contradiction, that some $\ell$ has $h_\ell(H)\le 1$.
Let $A=\ell^{-1}(1)$ and $B=\ell^{-1}(0)$. By the previous paragraph:
no edge of $H$ has both vertices in $A$ (else that edge is an
admissible $2$-set), so $A$ is an independent set of $H$; and no
non-edge of $H$ has both vertices in $B$, so $B$ is a clique of
$H$. Hence,
$
|A|\le\alpha(H)=2
\text{ and }
|B|\le\omega(H)=3,
$
whence $7=|V(H)|=|A|+|B|\le 5$, a contradiction. Therefore, 
$h_\ell(H)\ge2$ for every $\ell$.
\end{proof}
}

Note that $h_\ell$ is additive over disjoint unions of graphs and the minimum over $\ell$ splits over components, so $\foe$ (and likewise $\fo$) is additive over disjoint unions. Hence $c_{\rm o}$ and $c_{\rm oe}$ are already determined by connected graphs, and disjoint copies of $\overline{C_7}$ show that the ratio $2/7$ is attained by both $\fo(G)/|V(G)|$ and $\foe(G)/|V(G)|$ at every order divisible by~$7$.

Thus, we know that $1/5\le c_{\rm o}\le 2/7$ and $\YZ{2/11}\le \VV{c_{\rm oe}}\le 2/7$ \VV{(the upper bound on $c_{\rm oe}$ follows from Proposition~\ref{prop:lower}, and also from $c_{\rm oe}\le c_{\rm o}$, which holds since $\foe(G)\le \fo(G)$ for every graph $G$)}. What are the exact values of $c_{\rm o}$ and $c_{\rm oe}$?

A computer search (whose results we have not independently verified apart from Proposition~\ref{prop:lower}) reported that the minimum of $\foe(G)$ over all graphs $G$ with $\delta(G)\ge 1$ and $|V(G)|=n$ equals $\lceil 2n/7\rceil$ for every $n\le 10$, that $\overline{C_7}$ is the unique connected $7$-vertex graph $G$ with $\foe(G)=2$, and that no graph with a smaller ratio was found by heuristic searches on up to $18$ vertices. This suggests the conjecture that $c_{\rm o}=c_{\rm oe}=2/7$.

\vspace{1mm}

It would be interesting to \VV{improve} the upper bound on $t_2$ and to determine the exact values of $t_2, t_3$ and $t_4$.

	
\end{document}